\documentclass[12pt,english]{amsart}
\usepackage[T1]{fontenc}
\usepackage[latin9]{inputenc}
\usepackage{verbatim}
\usepackage{amsthm}
\usepackage{amstext}
\usepackage{amssymb}

\input xy
\xyoption{all}

\makeatletter

\numberwithin{equation}{section}
\numberwithin{figure}{section}
\theoremstyle{plain}
\newtheorem{thm}{Theorem}
\theoremstyle{plain}
\newtheorem{prop}[thm]{Proposition}
\theoremstyle{plain}
\newtheorem{lem}[thm]{Lemma}
\theoremstyle{plain}
\newtheorem{cor}[thm]{Corollary}
\newtheorem{rem}[thm]{Remark}
\newtheorem{ex}[thm]{Example}

\usepackage{amsthm}
\usepackage[active]{srcltx}

\makeatletter

\numberwithin{equation}{section}
\numberwithin{figure}{section}

\def\quot{/\!\!/}
\def\tr{\mathsf{tr}}

\def\hom{\mathsf{Hom}}
\def\C{\mathbb{C}}

\makeatother

\usepackage{babel}

\title[Parabolic Character Varieties of Free Groups]{The Topology
of Parabolic Character Varieties of Free Groups}

\author[I. Biswas]{Indranil Biswas}

\address{School of Mathematics, Tata Institute of Fundamental
Research, Homi Bhabha Road, Bombay 400005, India}

\email{indranil@math.tifr.res.in}

\author[C. Florentino]{Carlos Florentino}

\address{Departamento Matem\'atica, Instituto Superior T\'ecnico, Av. Rovisco
Pais, 1049-001 Lisbon, Portugal}

\email{cfloren@math.ist.utl.pt}

\author[S. Lawton]{Sean Lawton}

\address{Department of Mathematics, The University of Texas-Pan American,
1201 West University Drive Edinburg, TX 78539, USA}

\email{lawtonsd@utpa.edu}

\author[M. Logares]{Marina Logares}

\address{Instituto de Ciencias Matem\'aticas (CSIC-UAM-UC3M-UCM),
Nicol\'as Cabrera 15, Campus Cantoblanco UAM, 28049 Madrid, Spain}

\email{marina.logares@icmat.es}

\subjclass[2000]{14L30, 20E05, 14P25, 14L17}

\keywords{Free group, representation space, conjugacy class,
strong deformation retraction, complex algebraic group}

\makeatother

\usepackage{babel}

\begin{document}
\begin{abstract}
Let $G$ be a complex affine algebraic reductive group, and let
$K\,\subset\, G$ be
a maximal compact subgroup. Fix $\mathbf{h}\,:=\,(h_{1}\,,\cdots\,,h_{m})\,\in\, K^{m}$.
For $n\, \geq\, 0$, let $\mathsf{X}_{\mathbf{h},n}^{G}$ (respectively, 
$\mathsf{X}_{\mathbf{h},n}^{K}$)
be the space of equivalence classes of representations of the free
group on $m+n$ generators in $G$ (respectively, $K$) such that
for each $1\leq i\leq m$, the image of the $i$-th free generator
is conjugate to $h_{i}$. These spaces are parabolic analogues
of character varieties of free groups. We prove that 
$\mathsf{X}_{\mathbf{h},n}^{K}$
is a strong deformation retraction of $\mathsf{X}_{\mathbf{h},n}^{G}$.
In particular, $\mathsf{X}_{\mathbf{h},n}^{G}$ and $\mathsf{X}_{\mathbf{h},n}^{K}$
are homotopy equivalent. We also
describe explicit examples relating $\mathsf{X}_{\mathbf{h},n}^{G}$
to relative character varieties.
\end{abstract}
\maketitle

\section{Introduction}

\thispagestyle{empty}Let $G$ be a complex affine algebraic reductive
group. Fix a maximal compact subgroup $K$ of $G$.

Denote by $\hom (\mathsf{F}_{n},G)$ the space of representations in $G$ of 
the free group $\mathsf{F}_{n}$ of rank $n$.
The group $G$ acts on $\hom(\mathsf{F}_{n},G)$
by conjugation; the action of any $g\,\in\, G$ sends any
$\rho\,\in\,\hom(\mathsf{F}_{n},G)$
to the homomorphism defined by $\gamma\,\longmapsto\,
g^{-1}\rho(\gamma)g$,
where $\gamma\,\in\,\mathsf{F}_{n}$. The GIT quotient space \[
\mathsf{X}_{n}^{G}\,:=\,\hom(\mathsf{F}_{n},G)\quot G
\]
 is an affine algebraic variety, known as the $G$-\textit{character
variety} of $\mathsf{F}_{n}$.

Similarly, we have the compact topological space
\[
\mathsf{X}_{n}^{K}\,:=\,\hom(\mathsf{F}_{n},K)/K
\]
that parametrizes the equivalence classes of homomorphisms from
$\mathsf{F}_{n}$ to $K$. 

Whenever we will mention topology of $\mathsf{X}_{n}^{G}$, we will mean
the Euclidean topology of the underlying complex space.

In \cite{FL1} it was shown that there is a strong deformation retraction
from $\mathsf{X}_{n}^{G}$ to $\mathsf{X}_{n}^{K}$. In particular,
these spaces have the same homotopy type. On the other hand, if $\pi$
is the fundamental group of a closed orientable surface, then it is known
that the
corresponding spaces $\hom(\pi,G)\quot G$ and $\hom(\pi,K)/K$ do
not have the same homotopy type; in fact, any connected component
of $\hom(\pi,G)\quot G$ differs cohomologically from the
corresponding connected component of $\hom(\pi,K)/K$
\cite{BF}.

Our aim here is to consider the parabolic analogues of $\mathsf{X}_{n}^{G}$
and $\mathsf{X}_{n}^{K}$.

Take any integer $m\, \geq\, 0$, and fix
$(h_{1}\,,\cdots\,,h_{m})\,\in\, K^m$. Let $C_{j}^{G}$
(respectively, $C_{j}^{K}$) be the conjugacy class of $h_{j}$ in
$G$ (respectively, $K$). For any $n\, \geq\, 0$, let $\mathsf{F}_{m+n}$ be the free group on $m+n$ generators; the $i$-th generator of $\mathsf{F}_{m+n}$
will be denoted by $e_{i}$. Define
\[
\mathsf{H}_{\mathbf{h},n}^{G}\,:=\,
\{\rho\,\in\,\hom(\mathsf{F}_{m+n}\,,G)\, \mid\,\rho(e_{j})\,\in\,
C_{j}^{G}\, ,~\,\forall~1\leq j\leq m\}\]
 and \[
\mathsf{H}_{\mathbf{h},n}^{K}\,:=\,\{\rho\,\in\,
\hom(\mathsf{F}_{m+n}\,,K)\,\mid\,\rho(e_{j})\,\in\,
C_{j}^{K}\, ,~\,\forall~1\leq j\leq m\}\,.\]
As before, the group $G$ (respectively, $K$) acts on 
$\mathsf{H}_{\mathbf{h},n}^{G}$
(respectively, $\mathsf{H}_{\mathbf{h},n}^{K}$) via the conjugation
action of $G$ (respectively, $K$) on itself. Define \[
\mathsf{X}_{\mathbf{h},n}^{G}\,:=\, 
\mathsf{H}_{\mathbf{h},n}^{G}\quot G~\,\text{~and~}\, 
~\mathsf{X}_{\mathbf{h},n}^{K}\,:=\, 
\mathsf{H}_{\mathbf{h},n}^{K}/K\,.\]
Note that if $m=0$ these spaces reduce to the previous ones, $\mathsf{X}_{n}^{G}$ and $\mathsf{X}_{n}^{K}$, respectively.

We prove the following (see Corollary \ref{cor-main}):

\begin{thm}\label{thm-i}
The quotient space $\mathsf{X}_{\mathbf{h},n}^{K}$ is a strong 
deformation
retraction of $\mathsf{X}_{\mathbf{h},n}^{G}$. In particular, $\mathsf{X}_{\mathbf{h},n}^{G}$
and $\mathsf{X}_{\mathbf{h},n}^{K}$ are homotopy equivalent.
\end{thm}

Take closed subgroups $K_1\, ,\cdots\, ,K_m$ of $K$. Let $G_i$
be the Zariski closure of $K_i$ in $G$. Theorem \ref{thm-i} is
derived from the following more general result:

\begin{thm}
Let $\left(\left(\prod_{i=1}^mK/K_i\right)\times K^n\right)/K$
$($respectively, $\left(\left(\prod_{i=1}^mG/ G_i\right)\times 
G^n\right)\quot G)$ be the 
quotient for the diagonal left-translation action of $K$ $($respectively, $G)$ on 
the $m$ homogeneous spaces $K/K_i$ $($respectively, $G/G_i)$, and conjugation on the 
other $n$ factors of $K$ $($respectively, $G)$. Then 
$\left(\left(\prod_{i=1}^mK/K_i\right)\times K^n\right)/K$ is a strong
deformation retraction of $\left(\left(\prod_{i=1}^mG/ G_i\right)\times 
G^n\right)\quot G$.
\end{thm}

Let $\Sigma$ be the complement of $p$ points of a compact surface of 
genus $g$, with $p\, \geq\, 1$. The fundamental group of $\Sigma$ 
is isomorphic
to the free group $\mathsf{F}_{m+n}$ of rank $m+n=2g+p-1$. By choosing a free generating
set $\gamma=\left\{ \gamma_{1}\,,\cdots\,,\gamma_{m+n}\right\} $ of
$\pi_{1}(\Sigma,s_{0})$, where $s_{0}$ is a base point, the space of
representations of $\pi_{1}(\Sigma,s_{0})$ into $G$ gets identified with
$\hom(\mathsf{F}_{m+n},G)$ which is isomorphic to
$G^{m+n}$, via the evaluation map
$$
\mathsf{ev}_{\gamma}\,:\,\hom(\pi_{1}(\Sigma,s_{0}),G)\cong \hom(\mathsf{F}_{m+n},G)\,\longrightarrow\,
 G^{n+m}\,,~\,\rho\,\longmapsto\,(\rho(\gamma_{1})\,
,\cdots\,,\rho(\gamma_{m+n}))\, .
$$
Using this identification, we have
\[
\mathsf{X}_{\mathbf{h},n}^{G}\,\subset\, \hom(\pi_{1}(\Sigma,s_{0}),G)\quot G 
\, .
\]

In the case of $m=0$, by sending any flat connection on $\Sigma$ to its
monodromy representation, the space $\mathsf{X}_{n}^{K}$
(respectively, $\mathsf{X}_{n}^{G}$) gets
identified with the
moduli space of flat principal $K$--bundles (respectively,
completely reducible flat $G$--bundles) over $\Sigma$.
For $m>0$, we call $\mathsf{X}_{\mathbf{h},n}^{G}$ and
$\mathsf{X}_{\mathbf{h},n}^{K}$ parabolic character varieties,
because of their analogy with parabolic vector bundles.

As studied in \cite{BeGo} and \cite{La2},
there is a natural ``boundary'' map associating the conjugacy classes around the
punctures to a representation in $\hom(\pi_{1}(\Sigma,s_{0}),G)$. 
{\em Relative} character varieties are then defined as inverse images under the boundary map.
In the last two sections, we consider various examples of 
parabolic character varieties and show how they relate to relative character varieties.

\section{Deformation Retractions of Homogeneous Spaces}

The base field is the complex numbers. Let $G$ be an affine algebraic
reductive group. Fix a maximal compact subgroup $K\,\subset\, G$.
Let $X$
be an affine $G$--variety. The action of $G$ on $X$ will be denoted
by {}``$\cdot$'', meaning the action of $g\,\in\, G$ sends any $x\,\in\, X$
to $g\cdot x\,\in\, X$. The algebra of $G$-invariant regular functions
on $X$ will be denoted by $\mathbb{C}[X]^{G}$. Let \[
X\quot G:=\mbox{Spec}(\mathbb{C}[X]^{G})\]
 be the quotient, which is an affine variety \cite{Dol}. We call
$X\quot G$ the \textit{affine quotient} of $X$ by $G$. It satisfies the
usual properties of good categorical quotients (see \cite{Dol}).

In this situation, one can show that there is a pair $(V\,,\iota)$,
where $V$ is a $G$--module of finite (complex) dimension, and \[
\iota\,:\, X\,\hookrightarrow\, V\]
 is a $G$--equivariant embedding. Fix a $K$-invariant Hermitian
inner product $\langle-\, ,-\rangle$ on $V$. For any $x\,\in\, X$,
let \[
F_{x}\,:\, G\,\longrightarrow\,\mathbb{R}\]
 be the smooth function defined by $g\,\longmapsto\,||\iota(g\cdot x)||^{2}$.
With respect to these choices of $V$, $\iota$ and $\langle -\, , -
\rangle$, the Kempf--Ness set of $X$ is defined to be
\begin{equation}\label{kns}
{\rm KN}_{X}\,:=\,\{x\,\in\, X\,\mid\,\,\left(dF_{x}\right)_{e}\,
=\,0\}\, ,
\end{equation}
where $e\,\in\, G$ is the identity element. This implies that 
${\rm KN}_X$ is a real algebraic subspace, and hence it is a CW-complex (\cite{BCR, Hi}).

For any $x\,\in\, X$, let $K\cdot x$ denote the $K$-orbit passing
through $x$. We recall a theorem of Kempf and Ness.

\begin{thm}[{\cite[Theorem 0.1, Theorem 0.2]{KN}}]
\label{thm:KN} All the critical points of $F_{x}$
are minima. Moreover, $F_{x}$ has a critical point if and only if
$G\cdot x\, \subset\, X$ is closed. Also, if $x\,\in\,{\rm 
KN}_{X}$, then $K\cdot x\,\subset\,{\rm KN}_{X}$.
\end{thm}

A deformation retraction from a topological space $M$ to a subspace
$N\,\subset\, M$ is a homotopy $$\phi\, :\, [0,1]\times
M\,\longrightarrow\, M$$ between the identity map of $M$, and a 
retraction $r_N\,:\,M\,\longrightarrow\, N$.
It is called a strong deformation retraction whenever it fixes the points in 
$N$, that is, 
$\phi(t, x)\,=\, x$ for all $x\in N$ and $t\in [0,1]$ (see \cite{Spanier}).
To abbreviate, we often express the above situation by saying that
$N\,\hookrightarrow\, M$, i.e., the inclusion itself,
is (strong) a deformation retraction.

\begin{rem}\label{quotientsdr}
If a group $J$ acts on $M$ amd on $N$ such that there is a $J$--equivariant
homotopy equivalence $\alpha: N\to M$ (i.e., a $J$-equivariant map with a homotopy inverse and
whose corresponding homotopies are $J$-equivariant), then the induced
map $\widetilde{\alpha}:\, N/J \longrightarrow\, M/J$ is again a homotopy
equivalence. Indeed, a $J$--equivariant homotopy
$$
[0\, ,1]\times M\, \longrightarrow\, M
$$
descends to a homotopy
$$
[0\, ,1]\times (M/J)\, \longrightarrow\, (M/J)
$$
of quotient spaces.
\end{rem}

We now recall a theorem of Neeman and Schwarz.

\begin{thm}[{\cite[Theorem 2.1]{Ne}, \cite[Corollary 
4.7,Corollary 5.3]{Sch1}}]
\label{thm:Neeman-Schwarz} The composition \begin{equation}
{\rm KN}_{X}\,\hookrightarrow\, X\,\longrightarrow\, X\quot G\label{eq:Nee-Sch}\end{equation}
is proper and induces a homeomorphism ${\rm KN}_{X}/K\,\stackrel{\sim}{\longrightarrow}\, X\quot G$.
Moreover, there is a $K$--equivariant deformation retraction from
$X$ to ${\rm KN}_{X}$. Therefore, there is a deformation
retraction $$f\,:\,{\rm KN}_{X}/K\,\longrightarrow\, X/K\, .$$
\end{thm}

The above theorem will be used in getting a deformation retraction between
homogeneous spaces. We start with the following proposition which
is basically an adaptation of some results in \cite{FL2}
to our situation.

\begin{prop}
\label{pro:general-retraction} Suppose that we have the following
commutative
diagram of maps of CW complexes:
\[
\begin{array}{ccc}
Y & \stackrel{\alpha}{\rightarrow} & X\\
~\Big\downarrow f & & \Vert\\
Z & \stackrel{\beta}{\rightarrow} & X\end{array}\]
where $f$ is an inclusion of a subcomplex, and
the maps $\alpha$ and $\beta$ are homotopy equivalences.
Then $f$ is a strong deformation retraction.

\end{prop}

\begin{proof}
For any nonnegative integer $n$, let
\[
\alpha_{n}\,:\,\pi_{n}(Y)\,\longrightarrow\,
\pi_{n}(X)~\,\text{~and~}\,~\beta_{n}\,:\,\pi_{n}(Z)\,
\longrightarrow\,\pi_{n}(X)\]
be the homomorphisms induced by $\alpha$ and $\beta$ respectively.
Since $\alpha$ and $\beta$ are homotopy equivalences, we know
that $\alpha_{n}$ and $\beta_{n}$ are isomorphisms for all $n\,\geq\,0$.
Since $$\alpha\,=\,\beta\circ f\, ,$$ it follows that $f$ also induces isomorphisms on all homotopy groups. Then, from
Whitehead's Theorem (see \cite[Theorem 4.5, page 346]{Hat}) we conclude that
$Y$ is indeed a strong deformation retraction of $Z$. 
\end{proof}

\begin{rem}
\label{eqsdrrem}
Using the same notation as in Proposition \ref{pro:general-retraction}, 
suppose that a group $J$ acts on $X$, $Y$ and $Z$, 
in such a way that $f$ is a $J$-equivariant inclusion of a subcomplex, 
and $\alpha$ and $\beta$ are $J$--equivariant homotopy equivalences.
Then all maps in the diagram of Proposition
 \ref{pro:general-retraction} are $J$--equivariant, and by Remark \ref{quotientsdr}
we obtain a commutative diagram of maps of quotient spaces
\[
\begin{array}{ccc}
Y/J & \stackrel{\widetilde{\alpha}}{\longrightarrow} & X/J\\
\Big\downarrow\widetilde{f} & & \Vert\\
Z/J & \stackrel{\widetilde{\beta}}{\longrightarrow} & X/J
\end{array}\]
Thus, from Proposition \ref{pro:general-retraction}, the induced map $\widetilde{f}\,:\,Y/J\,\longrightarrow\,
Z/J$ is a strong deformation retraction, provided the quotient spaces $Y/J$ and $Z/J$ are CW complexes.
Note however that $f$ is not necessarily a $J$-equivariant strong deformation retract.
\end{rem}

Recall that $G$ is an affine algebraic reductive group, and $K$ is a maximal compact
subgroup of $G$. Suppose that $J$ is a closed subgroup of $K$. 
Then the Zariski closure $H$ of $J$ (in $G$) is a reductive 
subgroup of $G$ with $$J\,=\, H\cap K\, ,$$ and
$J$ is a maximal compact subgroup of $H$.

Consider the homogeneous space $G/H$. Since $H$ is reductive,
we can regard $G/H$ as an affine quotient of the $H$--variety $G$,
where $H$ acts by right translations on $G$. Note that $G/H\,=\,G\quot
H$ is a smooth affine variety. Similarly, we have the quotient $K/J$,
which is a compact manifold.

\begin{lem}
\label{le} There is a canonical inclusion of smooth manifolds $f\,:\,
K/J\,\longrightarrow\, G\quot H$.
\end{lem}

\begin{proof}
Let $f\,:\, K/J\,\longrightarrow\, G/H=G\quot H$ be the canonical
map that sends any coset $kJ$, $k\,\in\, K$, to the coset
$kH$. This $f$ is well-defined since $J\,\subset\, H$. Take any 
$k_1\, ,k_2\, \in\, K$ such that $k_{1}H\,=\, k_{2}H$.
To prove the lemma it suffices to show that
\begin{equation}\label{ls}
k_{1}J\,=\, k_{2}J\, .
\end{equation}

To prove \eqref{ls}, take $h\, :=\, k^{-1}_2 k_{1}\,\in \,H$.
Let
\begin{equation}\label{fp}
\text{Lie}(H)\,=\,\text{Lie}(J)\oplus\mathfrak{p}
\end{equation}
be the Cartan decomposition of the Lie algebra of $H$. Let
\[
h\,=\, j\exp{(p)}
\]
be the polar decomposition,
where $j\,\in\, J$ and $p\,\in\,\mathfrak{p}$. Then,
\begin{equation}
j^{-1}k_{2}^{-1}k_{1}\,=\,\exp{(p)}\,.\label{c3}
\end{equation}
 The left--hand side of \eqref{c3} is in the maximal compact subgroup
$K$ of $G$, so by the uniqueness of polar decomposition in $G$,
we have that \[
\exp{(p)}\,=e~\,\text{~and~}\,~k_{1}\,=\, k_{2}j\, .\]
Therefore, \eqref{ls} holds.
\end{proof}

\begin{lem}\label{biaction}
There is a $K\times K$-equivariant strong deformation retraction from $G$ to $K$. 
\end{lem}

\begin{proof}
The polar decomposition of $G$ is given by $K\times \mathfrak{p}$ and the diffeomorphism is given by $(k,p)\mapsto k\mathrm{exp}(p)$.
Then $\phi_t(k\mathrm{exp}(p))=k\mathrm{exp}(tp)$ is a strong deformation retraction. Consider $a,b\in K$.  Then 
\begin{eqnarray*}a\phi_t(k\mathrm{exp}(p))b^{-1}&=&ak\mathrm{exp}(tp)b^{-1}\\ &=&akb^{-1}\mathrm{exp}(t\mathrm{Ad}_b(p))\\ &=&\phi_t(akb^{-1}\mathrm{exp}(\mathrm{Ad}_b(p)))\\ &=&\phi_t(ak\mathrm{exp}(p)b^{-1}).\end{eqnarray*}
\end{proof}

\begin{thm}
\label{thm:SDR-Hom-sp}The inclusion $K/J \hookrightarrow G/H$ from Lemma \ref{le} is a strong deformation retraction.\end{thm}

\begin{proof}
In \eqref{kns}, set $X\,=\,G$ equipped with the
right--translation action of $H$. By the Kempf--Ness
construction, we obtain a commutative diagram of natural inclusions
\begin{equation}\label{ph}
\begin{array}{ccc}
K & \hookrightarrow & G\\
\Big\downarrow\varphi & & \Vert\\
{\rm KN}_{G} & \hookrightarrow & G.\end{array}
\end{equation}
As a special case of Lemma \ref{cwinc} below, $K\,\hookrightarrow\, {\rm KN}_{G}$.  

As semi-algebraic spaces, $K$ can be taken to be a sub-complex of the CW-complex ${\rm KN}_{G}$, as can the quotients $K/J\hookrightarrow {\rm KN}_{G}/J$ (\cite{BCR, Hi, Sch1}).

From Theorem \ref{thm:KN} we know that ${\rm KN}_{G}$ is closed under the right--translation action of $J\, :=\, H\bigcap K$. So all the 
maps in \eqref{ph} are $J$-equivariant. 

The top map in \eqref{ph} induces a right $J$-equivariant strong deformation retraction by Lemma \ref{biaction}.  The bottom map likewise induces a $J$-equivariant strong deformation retraction by Theorem \ref{thm:Neeman-Schwarz}.  Therefore, by Proposition \ref{pro:general-retraction} and Remark \ref{eqsdrrem}, the inclusion
$$\widetilde{\varphi}\, :\, K/J\,\longrightarrow\, {\rm KN}_{G}/J$$
given by $\varphi$ in \eqref{ph} is a strong deformation 
retraction. Finally, by the homeomorphism
${\rm KN}_{G}/J\,\cong \,G/ H$ in Theorem
\ref{thm:Neeman-Schwarz}, we obtain a strong deformation retraction
of $G/ H$ to $K/J$. 
\end{proof}

\begin{rem}\label{orbitremark}
Fix an element $k\,\in\, K$. Let
\[
G\cdot k\,:=\,\left\{ gkg^{-1}\, \mid\, g\in G\right\}
\]
be its orbit under the conjugation action of $G$. Let
\[
G_{k}\,:=\,\left\{ g\in G\,\mid\, gkg^{-1}=k\right\} \,
\]
be the centralizer. Since $k$ is semisimple, it follows that $G_k$ is
a complex reductive subgroup of $G$ \cite[\S~2.2, p. 26, Theorem]{Hu}.
Sending any $g\, \in\, G$ to $gkg^{-1}\, \in\,
G\cdot k$, we get an identification of $G/G_{k}$ with $G\cdot k$.
This identification takes the the left--translation action of
$G$ on $G/G_{k}$ to the adjoint action of $G$ on $G\cdot k$. Moreover, this 
mapping $G/G_{k}\, \longrightarrow\, G\cdot k$ defines an isomorphism of 
quasi-projective varieties \cite[p. 83]{hum}.

Similarly, the $K$--orbit $K\cdot k$ gets identified with $K/K_{k}$,
where $K_{k}$ is the centralizer of $k$ inside $K$. Note that
$K_{k}\,=\, G_k\bigcap K$.
\end{rem}

\begin{cor}
\label{cor:orbit-SDR} Fix any $k\,\in\, K$. The inclusion
of orbits $K\cdot k\,\subset\, G\cdot k$ is a strong deformation
retraction.
\end{cor}

\begin{proof}
The group $K_{k}$ is compact because it is closed in $K$.
Its Zariski closure in $G$ is $G_k$. In view of the canonical inclusion
$f\,:\, K/K_{k}\,\longrightarrow\, G/G_{k}$
in Lemma \ref{le}, we can apply Theorem \ref{thm:SDR-Hom-sp} with
$J\,=\, K_{k}$ and $H\,=\, G_{k}$. Therefore,
$K\cdot k$ is a strong deformation retraction of $G\cdot k$.
\end{proof}

Note that the retraction in Theorem \ref{thm:SDR-Hom-sp} is more 
general than that in Corollary \ref{cor:orbit-SDR}, since not 
every closed subgroup $J\subset K$ is the stabilizer of some 
$k\,\in\, K$.

\section{Parabolic Character Varieties of a Free Group}

Let $m\,\geq\, 0$ be an integer, and let $K_{i}$ for $1\leq i \leq m$ be
closed subgroups of the compact group $K$. For each $1\,\leq\, 
i\,\leq\, m$, let $G_i\, \subset\, G$ be the Zariski closure of
$K_i$; note that $G_i$ is reductive, and $\text{Lie}(G_i)\,=\,
\text{Lie}(K_i)\oplus \sqrt{-1}\cdot \text{Lie}(K_i)$. Define
\begin{equation}\label{cX}
\mathcal{X}_{m,n}:=\left(\prod_{i=1}^mG/G_{i}\right)\times G^{n}
\end{equation}
and
\begin{equation}\label{cY}
\mathcal{Y}_{m,n}:=\left(\prod_{i=1}^mK/K_{i}\right)\times K^{n}\, .
\end{equation}

Since $G\,=\,G/ G_{j}$ if $K_{j}=\{e\}$, the parameter $n$ could be removed; but it is retained because it will be useful for defining parabolic character varieties later on.  As usual, the zero-fold Cartesian product $X^0$ consists of only the empty tuple.  So $\mathcal{X}_{0,n}=G^n$ and $\mathcal{Y}_{0,n}=K^n$.

Let $K$ act as left--translations on the homogeneous spaces $K/K_{j}$ and $G/G_{j}$ for $1\leq j\leq m$; the group $K$ acts on 
$K$ and $G$ through inner automorphisms. These actions provide an action of $K$ on both $\mathcal{Y}_{m,n}$ and $\mathcal{X}_{m,n}$. Clearly, the map \begin{equation}\label{f}
f:\:\,\mathcal{Y}_{m,n}\,\hookrightarrow\,\mathcal{X}_{m,n}
\end{equation}(see \eqref{cX} and \eqref{cY}) is $K$-equivariant, and it is injective by Lemma \ref{le}. 

The $K$-equivariant strong deformation retraction in \cite{FL1}, and the strong deformation retraction in Theorem \ref{thm:SDR-Hom-sp} together prove the following proposition.

\begin{prop}\label{topsdr}
For any $m, n \,\geq\, 0$, there is a strong deformation retraction from $\mathcal{X}_{m,n}$ to $\mathcal{Y}_{m,n}$.
\end{prop}

The group $(\prod_{i=1}^mG_i)\times G$ acts on $G^m\times G^n$ as follows:  $$(h_1,...,h_m,g)\cdot (f_1,...,f_m,g_1,...,g_n)=(gf_1h_1^{-1},...,gf_mh_m^{-1},gg_1g^{-1},...,gg_ng^{-1}).$$  It is easy to see that the two GIT quotients $(G/G_1 \times \cdots\times G/G_m \times G^n)\quot G$ and $(G^m \times G^n)\quot(\prod_{i=1}^m G_i\times G)$ are isomorphic; likewise for the compact quotients $(K^m\times K^n)/(\prod_{i=1}^m K_i\times K)$ and  $\mathcal{Y}_{m,n}/K.$

One can apply the Kempf-Ness construction to the affine $( \prod_{i=1}^mG_i)\times G$--variety $G^m\times G^n$. 

\begin{lem}\label{cwinc}
$K^m\times K^n$ is a subcomplex of $\mathrm{KN}_{G^m\times G^n}$.
\end{lem}

\begin{proof}
The fact that it is a subset follows from the computation in the proof of Proposition A.1 in \cite{FL2}.  Since both sets are algebraic, $K^m\times K^n$ may be taken to be a subcomplex by \cite{Hi, BCR}.
\end{proof}

Therefore, from the above lemma we deduce the following:

\begin{thm}
\label{thm:product-sdr}There is a strong deformation retraction from
$\mathcal{X}_{m,n}\quot G$ to $\mathcal{Y}_{m,n}/K$.
\end{thm}

\begin{proof}
The proof is analogous to that of Theorem \ref{thm:SDR-Hom-sp}.
We have a commutative diagram of $(\prod_{i=1}^m K_i)\times K$--equivariant inclusions\[
\begin{array}{ccc}
K^m\times K^n& \stackrel{\alpha}{\hookrightarrow} & G^m\times G^n\\
\Big\downarrow\varphi & & \Vert\\
{\rm KN}_{G^m\times G^n} & \stackrel{\beta}{\hookrightarrow}
& G^m\times G^n\end{array}\]

The bottom map corresponds to a $(\prod_{i=1}^m K_i)\times K$-equivariant strong deformation retraction by Theorem \ref{thm:Neeman-Schwarz}, the top map corresponds to a $(\prod_{i=1}^m K_i)\times K$-equivariant strong deformation retraction by Lemma \ref{biaction}, and Lemma \ref{cwinc} gives the inclusion of complexes $\varphi$ that descends to a map on quotients that are themselves complexes \cite{Hi,BCR, Sch1}.

Then, by Proposition \ref{pro:general-retraction} and Remark \ref{eqsdrrem}, one obtains a strong deformation retraction from $$ \mathcal{X}_{m,n}\quot G\cong (G^m\times G^n)\quot (\prod_{i=1}^m G_i\times G)\cong {\rm KN}_{G^m\times G^n}/(K\times \prod_{i=1}^m K_i)$$
to the compact quotient $$(K^m\times K^n)/(\prod_{i=1}^m K_i\times K)\cong \mathcal{Y}_{m,n}/K.$$
\end{proof}

Finally, we apply Theorem \ref{thm:product-sdr} to the
parabolic character varieties. Fix
an element \[
\mathbf{h}\,:=\,(h_{1}\,,\cdots\,,h_{m})\,\in\, K^{m}\,.\]
 For $1\leq j \leq m$, let \begin{align*}
C_{j}^{K} & :=\, K\cdot h_{j}\,=\,\left\{ kh_{j}k^{-1}\,\mid\,\,
k\in K\right\} ~\, \text{and}\\
C_{j}^{G} & :=\, G\cdot h_{j}\,=\,\left\{ gh_{j}g^{-1}\,\mid\,\, g\in G\right\} \end{align*}
be the conjugation orbits in $K$ and $G$ respectively. Let $n\,\geq\,0$
be an integer.
Define
\[
\mathsf{H}_{\mathbf{h},n}^{G}\,:=\,
\left\{ (g_{1},\cdots,g_{m+n})\in G^{m+n}\,\mid\,
g_{j}\in C_{j}^{G},\, j=1,\cdots,m\right\} \,=\,
C_{1}^{G}\times\cdots\times C_{m}^{G}\times G^{n}\, .\]
Consider the diagonal action of $G$ on $G^{m+n}$ constructed
using the adjoint action of $G$ on itself. This action clearly
preserves the subset $\mathsf{H}_{\mathbf{h},n}^{G}\,\subset\, G^{m+n}$.
The restriction of this action to the
subgroup $K\, \subset\, G$ preserves the subset
$$
\mathsf{H}_{\mathbf{h},n}^{K}\,:=\,
\left\{ (g_{1},\cdots,g_{m+n})\in K^{m+n}\,\mid\,
g_{j}\in C_{j}^{K},\, j=1,\cdots,m\right\} \,=\,
C_{1}^{K}\times\cdots\times C_{m}^{K}\times K^{n}
$$
of $\mathsf{H}_{\mathbf{h},n}^{G}$.

We have the affine algebraic quotient\[
\mathsf{X}_{\mathbf{h},n}^{G}\,:=\,\mathsf{H}_{\mathbf{h},n}^{G}\quot G\]
 and the compact quotient \[
\mathsf{X}_{\mathbf{h},n}^{K}\,:=\,\mathsf{H}_{\mathbf{h},n}^{K}/K\,,\]
which are called the \textit{parabolic character varieties}.

\begin{cor}
\label{cor-main} Take integers $m\,\geq\, 0$ and $n\, \geq\, 0$.
For any $\mathbf{h}\,\in\, K^{m}$,
there is a strong deformation retraction of $\mathsf{X}_{\mathbf{h},n}^{G}$
onto $\mathsf{X}_{\mathbf{h},n}^{K}$.
\end{cor}

\begin{proof}
If $m=0$ the result follows from \cite{FL1}.  Now assume $m\geq 1$.
Let $\mathbf{h}\,:=\,(h_{1}\,,\cdots\,,h_{m})\,\in\, K^{m}$.
The centralizer of $h_j$ in $G$ (respectively, $K$) will be
denoted by $G_j$ (respectively, $K_j$); so $G_j$ is the
Zariski closure of $K_j$ in $G$, and $\text{Lie}(G_j)\,=\,
\text{Lie}(K_j)\oplus \sqrt{-1}\cdot \text{Lie}(K_j)$. As noted before in Remark \ref{orbitremark},
$G/G_j$ (respectively, $K/K_j$) is identified with the orbit
$G\cdot h_j$ (respectively, $K\cdot h_j$) for the adjoint action.
These give identifications
\[
\mathsf{H}_{\mathbf{h},n}^{G}\,\cong\,\mathcal{X}_{m,n}\]
and\[
\mathsf{H}_{\mathbf{h},n}^{K}\,\cong\, \mathcal{Y}_{m,n}\, ,\]
where the space $\mathcal{Y}_{m,n}$ (respectively,
$\mathcal{X}_{m,n}$) is constructed as in \eqref{cY}
(respectively, \eqref{cX}) using the subgroups
$K_{1},\cdots,K_{m}$ (respectively,
$G_{1},\cdots,G_{m}$). Under these identifications,
the diagonal conjugation action of $K$ on $\mathsf{H}_{\mathbf{h},n}^{K}$
becomes the diagonal action for the left--translation action on the
homogeneous spaces $K/K_{j}$ and the conjugation action on
the factors $K$. A similar statement holds for the actions of $G$
on $\mathcal{X}_{m,n}$ and $\mathsf{H}_{\mathbf{h},n}^{G}$. Therefore, the strong deformation
retraction of $\mathsf{X}_{\mathbf{h},n}^{G}\,:=\,\mathsf{H}_{\mathbf{h},n}^{G}\quot G$
to $\mathsf{X}_{\mathbf{h},n}^{K}\,:=\,\mathsf{H}_{\mathbf{h},n}^{K}/K$
is a direct consequence of Theorem \ref{thm:product-sdr}.
\end{proof}

\section{The generic case of parabolic character varieties}

In this section we discuss some examples, and the generic case, that is the case
when $\mathbf{h}=(x_{1},x_{2},\cdots,x_{m})\in K^{m}$ has a regular component.

\subsection{Elementary examples}

First we consider the case where $m=0$. In this case, 
as there is no parabolic data (no conjugation classes), we get back the 
deformation
retraction from the $G$--character variety to the $K$--character
variety of the free group of rank $n$: \[
\mathsf{X}_{n}^{K}\,=\,\hom(\mathsf{F}_{n},K)/K\:\hookrightarrow\:\mathsf{X}_{n}^{G}\,=\,\hom(\mathsf{F}_{n},G)\quot G,\]
that was obtained in \cite{FL1}.
Note in particular, for $n=1$,
that we have the inclusion\begin{equation}
K/K\hookrightarrow G\quot G.\label{eq:KandGclasses}\end{equation}

For a rather trivial example, let us consider $m=1$ and $n=0$. In
this case the vector $\mathbf{h}\in K^{m}$ is given by a single element
$x\in K$. Then, the parabolic representation spaces are single orbits\[
\mathsf{H}_{x,0}^{K}=K\cdot x\]
and \[
\mathsf{H}_{x,0}^{G}\,=\,G\cdot x\, .\]
Since the $K$-action (respectively, the $G$-action) is transitive on these
orbits, the parabolic character varieties consist of a single point:\[
\mathsf{X}_{x,0}^{G}=\,\mathsf{H}_{x,0}^{G}\quot 
G=\,\mathsf{H}_{x,0}^{K}/K=\, \mathsf{X}_{x,0}^{K}\, .\]

\subsection{The case $m=1$}
\label{sect4.1}

Let us consider arbitrary
$n$, a compact group $K$, and again $m=1$, so we are dealing with
only one conjugacy class. Since every element $x\in K$ lies in a
maximal torus of $K$, and all maximal tori of $K$ are conjugate,
it is no loss of generality to assume that $x\in T$, for some fixed
maximal torus $T\subset K$.

Let us now consider the {}``generic'' situation: assume that $x\in T$
is regular, that is $x$ is not fixed by any non-trivial element of
the Weyl group. Such regular elements form a dense set inside $T$,
and moreover it is known that, for regular $x$, the centralizer of
$x$ in $K$ is just the torus: $K_{x}=T$.

In this situation, we have: \[
\mathsf{H}_{x,n}^{K}=(K\cdot x)\times K^{n}\]
and the parabolic character variety can be described as follows. Consider
the action of $T\subset K$ on $K^{n}$ by simultaneous conjugation
and the corresponding quotient space $K^{n}/T$.

\begin{prop}
\label{pro:generic-m1}When $x\in T\subset K$ is regular, there is
a natural isomorphism of quotient spaces\[
\mathsf{X}_{x,n}^{K}=\,\mathsf{H}_{x,n}^{K}/K\quad\stackrel{\sim}{\longrightarrow}\quad K^{n}/T.\]
\end{prop}

\begin{proof}
We will describe the isomorphism explicitly. Consider the map:\begin{eqnarray*}
\eta\ :\ K^{n}/T & \longrightarrow & \mathsf{H}_{x,n}^{K}/K\\
{}[(k_{1},\cdots,k_{n})] & \longmapsto & [(x,k_{1},\cdots,k_{n})].\end{eqnarray*}
The map $\eta$ is well defined, since if $(k_{1},\cdots,k_{n})$
and $(h_{1},\cdots,h_{n})$ represent the same class in $K^{n}/T$,
then there is $t\in T$ such that $tk_{j}t^{-1}=h_{j}$ for all $1\leq j \leq n$
and because $x\,=\,txt^{-1}$, the elements $(x,h_{1},\cdots,h_{n})$
and $(x,k_{1},\cdots,k_{n})$ in $\mathsf{H}_{x,n}^{K}$ represents
the same class. The map $\eta$ is surjective: if $(y,h_{1},\cdots,h_{n})\in\mathsf{H}_{x,n}^{K}=(K\cdot x)\times K^{n}$
is arbitrary, since $y\in K\cdot x$, there exists $k\in K$ such
that $x=kyk^{-1}$, so \[
\eta(kh_{1}k^{-1},\cdots,kh_{n}k^{-1})\, =\,
[(x,kh_{1}k^{-1},\cdots,kh_{n}k^{-1})]=[(y,h_{1},\cdots,h_{n})].\]
Finally, $\eta$ is also injective, because if $[(x,k_{1},\cdots,k_{n})]=[(x,h_{1},\cdots,h_{n})]$,
then there is an element $k\in K$ such that $kxk^{-1}=x$ and $kk_{j}k^{-1}=h_{j}$
for all $1\leq j \leq n$. But since $K_{x}=T$ we conclude that $k\in T$
and so $[(k_{1},\cdots,k_{n})]=[(h_{1},\cdots,h_{n})]$.
\end{proof}

The same argument allows us to show that we have a natural isomorphism\[
\mathsf{X}_{x,n}^{G}=\,\mathsf{H}_{x,n}^{G}\quot G\quad\stackrel{\sim}{\longrightarrow}\quad G^{n}\quot T_{\mathbb{C}}\]
whenever we have $G_{x}=T_{\mathbb{C}}$ and $x\in T\subset T_{\mathbb{C}}\subset G$,
for the maximal torus $T_{\mathbb{C}}$ of $G$ which is the complexification
of $T$.

\subsection{The generic case when $m>1$}

The case when $m>1$ is not so easy to describe in general, but the
{}``generic'' case can also be written explicitly as a quotient
by an abelian group of the $m-1$ parabolic character variety. In
fact, we have:
\begin{thm}
Let $x_{1}\in T\subset K$ be a regular element. Let $\mathbf{h}=(x_{1},x_{2},\cdots,x_{m})\in K^{m}$
and $\mathbf{h}'=(x_{2},\cdots,x_{m})\in K^{m-1}$. Then, there are
natural isomorphisms of quotient spaces\[
\mathsf{X}_{\mathbf{h},n}^{K}=\,\mathsf{H}_{\mathbf{h},n}^{K}/K\quad\stackrel{\sim}{\longrightarrow}\quad\mathsf{H}_{\mathbf{h}',n}^{K}/T\]
and\[
\mathsf{X}_{\mathbf{h},n}^{G}=\,\mathsf{H}_{\mathbf{h},n}^{G}\quot G\quad\stackrel{\sim}{\longrightarrow}\quad\mathsf{H}_{\mathbf{h}',n}^{G}\quot T_{\mathbb{C}}.\]
\end{thm}

\begin{proof}
As in the proof of Proposition \ref{pro:generic-m1}, we can construct
explicit isomorphisms:\begin{eqnarray*}
\eta\ :\ \mathsf{H}_{\mathbf{h}',n}^{K}/T & \longrightarrow & \mathsf{H}_{\mathbf{h},n}^{K}/K\\
{}[(x_{2},\cdots,x_{n},k_{1},\cdots,k_{n})] & \longmapsto & [(x_{1},x_{2},\cdots,x_{n},k_{1},\cdots,k_{n})],\end{eqnarray*}
and an analogous map for $\mathsf{X}_{\mathbf{h},n}^{G}=\,\mathsf{H}_{\mathbf{h},n}^{G}\quot G$.
The proof that $\eta$ is well defined and bijective is the
same as before.
\end{proof}

\begin{rem}
{\rm The above theorem remains true if any one of $x_i$ in $(x_1,\cdots 
,x_m)$ is regular.}
\end{rem}

\section{Relative versus Parabolic Character Varieties}
\label{rcv-pcv}

Let $b\geq1$ be an integer and $\Sigma$ be a genus $g$ surface
with $b$ points removed. Then there is a presentation of the fundamental group of $\Sigma$ as
\begin{equation}
\label{presentation}
\pi_{1}(\Sigma)\cong\langle\alpha_{1},\beta_{1},\cdots,\alpha_{g},\beta_{g},\gamma_{1},\cdots,\gamma_{p}\ |\ \prod_{i=1}^{g}[\alpha_{i},\beta_{i}]\prod_{j=1}^{b}\gamma_{j}=1\rangle,
\end{equation}
where each $\gamma_{j}$ corresponds to a loop around a puncture. Note that
$\pi_{1}(\Sigma)$ is a free group of rank $2g+b-1$.
As discussed in \cite{BeGo} and \cite{La2}, there is a boundary map 
\[
\partial:\hom(\pi_{1}(\Sigma),G)\quot G\to(G\quot G)^{b}\]
 given by 
 \begin{equation}
\partial([\![\rho]\!])=([\![\rho(\gamma_{1})]\!],\cdots,[\![\rho(\gamma_{b})]\!]),\label{boundary}
\end{equation}
where $[\![\rho]\!]$ denotes the equivalence class of $\rho\in\hom(\pi_{1}(\Sigma),G)$. 

In general, $G\quot G$ is an irreducible affine variety isomorphic
to $T\quot W$ where $T$ is a maximal torus and $W$ is the Weyl
group (see \cite{Steinberg}). For a given $\mathbf{b}$ in the
image of $\partial$, the affine variety $\mathsf{X}_{rel}^{\mathbf{b}}:=\partial^{-1}(\mathbf{b})$
is called a \textit{relative character variety}.

Relative character varieties are in one-to-one correspondence with
framed $G$-local systems over $\Sigma$ as described by Fock and
Goncharov in \cite[Section 2, page 39]{FG}, which in turn are studied
in recent work by Biquard, Garc\'{i}a-Prada and Mundet i Riera in
relation with the moduli spaces of parabolic $G$-Higgs bundles \cite{BiGPM}.
This latter work aims at generalizing Simpson's correspondence
\cite{Si1} between $\mathrm{GL}(\ell,\mathbb{C})$-local
systems on $\Sigma$ with fixed monodromy in $U(\ell)$ around the
punctures and parabolic Higgs bundles, to any reductive Lie group
$G$ and any fixed monodromy in $G$. (This was done earlier for classical 
groups in \cite{BS}, and it was done for finite order monodromies around
punctures in \cite{Bi}.)

The parabolic character varieties we consider in this paper can also
be obtained in a similar way. Indeed, let $\mathsf{F}_{m+n}$ be the
free group with generators $e_{1},\cdots ,e_{m+n}$ as before, and consider
the map
\[
\partial_{par}:\hom(\mathsf{F}_{m+n},G)\quot G\cong G^{m+n}\quot G\to(G\quot G)^{m}
\]
 given by 
 \begin{equation}
 \label{delta-par}
\partial_{par}([\![\rho]\!])=([\![\rho(e_{1})]\!],\cdots,[\![\rho(e_{m})]\!]).
\end{equation}
Let $q:G^m\to (G\quot G)^m$ be the canonical quotient map.
For a given $\mathbf{h}=(h_{1},\cdots ,h_{m})\in G^m$, the parabolic
character variety $\mathsf{X}_{\mathbf{h},n}^{G}\subset\hom(\mathsf{F}_{m+n},G)\quot G$
is obtained by sending the first $m$ generators $e_{1},\cdots ,e_{m}$
to the conjugacy classes of $h_{1},\cdots ,h_{m}$. So, we have the 
identification
\[
\mathsf{X}_{\mathbf{h},n}^{G}=\partial_{par}^{-1}(q(\mathbf{h})).
\]
Note that we required the parabolic data $\mathbf{h}\in K^{m}$ for the deformation retraction in Theorem \ref{thm-i} to work. However, we can consider here the more general situation $\mathbf{h}\in G^{m}$.
Let $Z(G)$ be the center of $G$.

\begin{prop} For $n+m\geq 2$, the parabolic character variety $\mathsf{X}^{G}_{\mathbf{h},n}$
is an affine variety such that
$$
\dim \mathsf{X}^{G}_{\mathbf{h},n}\geq (n+m-1)\left(\dim G\right)+\dim Z(G) 
-m\left(\mathrm{rank}\,G\right)
$$
with equality holding for almost every choice of $\mathbf{h}$.
\end{prop}

\begin{proof}
Since $n+m\geq 2$, and only $Z(G)$ acts trivially on $G^{n+m}$,
we conclude $d_1:=\dim G^{n+m}\quot G=(n+m-1)\dim G+\dim 
Z(G)$. Since $G\quot G\cong T\quot W$, we have $d_2:=\dim (G\quot 
G)^{m}=m\left(\mathrm{rank}\, G\right)$. The map $\partial_{par}$ is always 
surjective, hence $\dim 
\mathsf{X}^{G}_{\mathbf{h},n}=\partial_{par}^{-1}\left(q(\mathbf{h}))\right)\geq 
d_1-d_2$ for all points $q(\mathbf{h})\in (G\quot G)^{m}$ and with equality 
holding on an open dense subset $U$ of $(G\quot G)^{m}$ \cite[Section 6, 
Theorem 7]{Sh}. Then equality holds for almost all $\mathbf{h}$ since the 
projection $q:G^{m}\longrightarrow (G\quot G)^{m}$ is surjective and 
continuous, and thus $q^{-1}(U)$ is also (Zariski) open and non-empty, and hence dense.
\end{proof}

We now relate the relative character varieties to the parabolic character
varieties. 
Note, however, that relative character varieties depend on the homeomorphism class of the underlying surface. 
On the other hand, character varieties and parabolic character varieties only depend on the free group
$\mathsf{F}_{m+n}$. 
Because of this, we need a concrete identification of the surface group 
$\pi_{1}(\Sigma)$ with a free group, that works for any genus
$g$ surface $\Sigma$ with $b$ points removed.

Using the presentation of $\pi_{1}(\Sigma)$ as in (\ref{presentation}), and supposing
we have $b>m\geq0$, $n,g\geq0$ and $m+n=2g+b-1$, define
the following isomorphism of free groups:
\begin{eqnarray}
\label{iso-free-gps}
\varphi:\ \mathsf{F}_{m+n} & \longrightarrow & \pi_{1}(\Sigma)\\
e_{i} & \longmapsto & \varphi(e_{i}) \nonumber
\end{eqnarray}
given by 
\[
\varphi(e_{i}):=\begin{cases}
\gamma_{i}, & i=1,\cdots ,b-1\\
\alpha_{1+i-b}, & i=b,\cdots ,g+b-1\\
\beta_{1+i-b-g},\ & i=b+g,\cdots ,2g+b-1.\end{cases}
\]

One easily checks that $\varphi$ is an isomorphism because
$$\{\alpha_{1},\cdots 
,\alpha_{g},\beta_{1},\cdots ,\beta_{g},\gamma_{1},\cdots ,\gamma_{b-1}\}$$
is a sequence of free generators of $\pi_{1}(\Sigma)$,
and $\gamma_b$ is expressed as
$$\gamma_b\, = \,
\left(\prod_{i=1}^g[\alpha_i,\beta_i]\prod_{j=1}^{b-1}\gamma_j\right)^{-1}\, 
.$$ 
Suppose that $\rho\in\hom(\pi_{1}(\Sigma),G)$. By composing with $\varphi$,
we obtain a representation $\rho\circ\varphi\in\hom(\mathsf{F}_{m+n},G)$,
and this defines an isomorphism
\[
\varphi^{*}:\hom(\pi_{1}(\Sigma),G)\quot G\to\hom(\mathsf{F}_{m+n},G)\quot G
\]
since this composition is equivariant with respect to the conjugation
action on the representation spaces.

\begin{prop}
\label{prop-diagram}Let $b>m\geq0$.
There is a commutative diagram
\[
\begin{array}{ccc}
\hom(\pi_{1}(\Sigma),G)\quot G & \stackrel{\partial}{\to} & (G\quot G)^{b}\\
\varphi^{*}\downarrow & & \downarrow\pi\\
\hom(\mathsf{F}_{m+n},G)\quot G & \stackrel{\partial_{par}}{\to} & (G\quot G)^{m}\end{array}\]
where $\pi$ is the projection onto the first $m$ factors.
\end{prop}
\begin{proof}
Let $\rho\in\hom(\pi_{1}(\Sigma),G)$. The boundary of $\Sigma$ corresponds
to the loops $\gamma_{1},\cdots ,\gamma_{b}$. Since $b>m$, commutativity 
follows from the computations
\[
\pi(\partial([\![\rho)]\!])=\pi([\![\rho(\gamma{}_{1})]\!],\cdots,[\![\rho(\gamma_{b})]\!])=([\![\rho(\gamma{}_{1})]\!],\cdots,[\![\rho(\gamma_{m})]\!])
\]
and 
\[
\partial_{par}([\![\varphi^{*}(\rho)]\!])=([\![(\varphi^{*}\rho)(e_{1})]\!],\cdots,[\![(\varphi^{*}\rho)(e_{m})]\!])=([\![\rho(\gamma{}_{1}))]\!],\cdots,[\![\rho(\gamma_{m})]\!]).
\]
This completes the proof.
\end{proof}

\begin{cor}
Let $\mathbf{h}=(h_{1},\cdots ,h_{m})\in G^{m}$ and let 
$\mathbf{b}\in\pi^{-1}(q(\mathbf{h}))$.
Then, with the notation as above, $\mathsf{X}_{rel}^{\mathbf{b}}\subset\mathsf{X}_{\mathbf{h},n}^{G}$
as affine sub-varieties. In fact,
\[
\mathsf{X}_{\mathbf{h},n}^{G}=\bigcup_{\mathbf{b}\in\pi^{-1}(q(\mathbf{h}))}\mathsf{X}_{rel}^{\mathbf{b}}.\]
\end{cor}
\begin{proof}
From the definitions, and from commutativity of the diagram one can
write
\[
\mathsf{X}_{\mathbf{h},n}^{G}=\partial_{par}^{-1}(\mathbf{h})\,=\, 
\partial^{-1}(\pi^{-1}(q(\mathbf{h})))\,=\, 
\bigcup_{\mathbf{b}\in\pi^{-1}(q(\mathbf{h}))}\partial^{-1}(\mathbf{b})\,=\, 
\bigcup_{\mathbf{b}\in\pi^{-1}(q(\mathbf{h}))}\mathsf{X}_{rel}^{\mathbf{b}},\]
completing the proof.
\end{proof}

\subsection{Examples}

In the following examples we consider $g=0$. Thus $m+n=b-1$, and $\Sigma$ is a sphere with $b=m+n+1$ punctures.
In particular, our restriction $b>m$ is automatically satisfied.
Now, the identification of $\mathsf{F}_{m+n}$ with $\pi_{1}(\Sigma)$ under 
equation (\ref{iso-free-gps}) becomes simply $\varphi(e_{i}):=\gamma_{i},
i=1,\cdots ,m+n$, and such that the last loop verifies 
$\gamma_b=\left(\gamma_1\cdots \gamma_{b-1}\right)^{-1}$. We will also restrict to the case $n=0$, so that $m=b-1$ becomes the rank of these free groups.

Under the isomorphisms
$\hom(\mathsf{F}_{m},G)\quot G \cong G^{m}\quot G \cong \hom(\pi_{1}(\Sigma),G)\quot G$, we can then write the diagram of Proposition
\ref{prop-diagram} in the following way:

$$
\xymatrix{
G^{m}\quot G\ar[r]^{\partial}\ar[rd]_{\partial_{par}} & (G\quot G)^{m+1}\ar[d]^{\pi}\\
& (G\quot G)^{m}
}
$$
where $\pi$ is the projection onto the first $m$ factors, and according to \eqref{boundary} and \eqref{delta-par}, we can write
for a given $(g_1,\cdots,g_{m})\in G^{m}$
$$\partial_{par}([\![g_1,\cdots,g_{m}]\!])=([\![g_1]\!], \cdots, [\![g_m]\!]),$$ and 
$$\partial([\![g_1,\cdots,g_{m}]\!])= ([\![g_1]\!], \cdots, [\![g_m]\!], [\![\left(g_1\cdots g_{m}\right)^{-1})]\!].$$

We concentrate on the lower rank simple groups $G=\mathrm{SL}(\ell,\C)$, as these allow for explicit
descriptions. For $G=\mathrm{SL}(\ell,\C)$, a generating set for the the coordinate ring 
$\C[G\quot G]=\C[G]^G$ is given by the 
coefficients of the characteristic polynomial of a generic $X\in 
\mathrm{SL}(\ell,\mathbb{C})$. In particular, we have $G\quot G \cong \mathbb{C}^{\ell-1}$.

As the $m=1$ cases were described in
Section \ref{sect4.1}, we consider now $m>1$.

\begin{ex}
{\rm
\label{ex2} Let $G=\mathrm{SL}(2,\mathbb{C})$. Then $G\quot G$
can be naturally identified with $\mathbb{C}$ by $[\![g]\!]\longmapsto\tr (g)$, 
for $[\![g]\!]\in G\quot G$.
Now let $m=2$. Then 
\[
\hom(\mathsf{F}_{2},G)\quot G=\mathrm{SL}(2,\C)^{2}\quot\mathrm{SL}(2,\C)\cong\C^{3}\]
by the Fricke-Klein-Vogt Theorem \cite{FK,V} and the boundary map
is an isomorphism from $G^{2}\quot G$ to $\C^{3}$ (see \cite[Theorem 2.1.1]{BeGo}),
which is explicitly given by 
\[
\partial([\![(g_{1},g_{2})]\!])=(\tr(g_{1}),\tr(g_{2}),\tr(g_{1}g_{2})),\]
since $\tr(g_1g_2)^{-1}=\tr(g_1g_2)$ in $\mathrm{SL}(2,\C)$.
Then the diagram for this example is 
\[
\xymatrix{G^{2}\quot G\cong\mathbb{C}^{3}\ar[r]^{\partial}\ar[rd]_{\partial_{par}} & (G\quot G)^{3}\cong\mathbb{C}^{3}\ar[d]^{\pi}\\
 & (G\quot G)^{2}\cong\mathbb{C}^{2},}
\]
where $\partial_{par}([\![g_{1},g_{2}]\!])=(\tr(g_{1}),\tr(g_{2}))$.
Thus, taking $K\,=\,\mathrm{SU}(2)$ and $\mathbf{h}\,=\,(h_{1},h_{2})\,\in\, 
\mathrm{SU}(2)^{2}$
we have, \[
\mathsf{X}_{\mathbf{h},0}^{G}=\partial_{par}^{-1}(\tr(h_{1}),\tr(h_{2}))\,=\,\partial^{-1}\left(\pi^{-1}(\tr(h_{1}),\tr(h_{2}))\right)=\mathbb{\partial}^{-1}(\mathbb{C})\cong\mathbb{C}\,,\]
where we used the isomorphism $\partial:G^{2}\quot 
G\stackrel{\sim}{\longrightarrow}\mathbb{C}^{3}$
and the obvious fact that $\pi^{-1}(z,w)\cong\mathbb{C}$ for any
$z,w\in\mathbb{C}$. }
\end{ex}

\begin{ex}
{\rm
Let again $G=\mathrm{SL}(2,\mathbb{C})$, but set $m=3$. In this
case, $\hom(F_{3},G)\quot G$ can be identified with a cubic hyper-surface
in $\mathbb{C}^{7}$ (a branched double cover over $\C^{6}$) in the
following way (see \cite[Section 3.B]{BeGo} for details).

First identify $\rho\in\hom(\mathsf{F}_{3},G)$ with $(\rho(e_{1}),\rho(e_{2}),\rho(e_{3}))=(g_{1},g_{2},g_{3})\in G^{3}$
and define $a=\tr(g_{1})$, $b=\tr(g_{2})$, $c=\tr(g_{3})$, $d=\tr(g_{1}g_{2}g_{3})$,
$x=\tr(g_{1}g_{2})$, $y=\tr(g_{2}g_{3})$ and $z=\tr(g_{3}g_{1})$.
With these seven traces as coordinates of $\mathbb{C}^{7}$, the 
quotient $G^{3}\quot G=\hom(F_{3},G)\quot G$
becomes the subvariety of ${\mathbb C}^7$, where the following cubic polynomial 
vanishes\[
x^{2}+y^{2}+z^{2}+xyz-(ab+cd)x-(ad+bc)y-(ac+bd)z-4+a^{2}+b^{2}+c^{2}+d^{2}-abcd.\]
 Denote this polynomial by $p(a,b,c,d,x,y,z)$. Again, identifying
$G\quot G$ with $\mathbb{C}$ using the trace, we have the commutative diagram
\[
\xymatrix{\mathbb{C}^{7}\supset G^{3}\quot G\ar[r]^{\partial}\ar[rd]_{\partial_{par}} & (G\quot G)\cong\mathbb{C}^{4}\ar[d]^{\pi}\\
 & (G\quot G)^{3}\cong\mathbb{C}^{3}}
\]
As in the previous example, $\partial_{par}(g_{1},g_{2},g_{3})=(a,b,c)$,
but now $\partial(g_{1},g_{2},g_{3})=(a,b,c,d)$ since in $\mathsf{F}_{3}$
we have $e_{0}=(e_{1}e_{2}e_{3})^{-1}$, and $\tr(g_{1}g_{2}g_{3})^{-1}=\tr(g_{1}g_{2}g_{3})$
in $\mathrm{SL}(2,\C)$. The relative character variety results from
fixing $\mathbf{b}\,=\,(a,b,c,d)\in\mathbb{C}^{4}$ (corresponding to the four
boundaries),
so\[
\mathsf{X}_{rel}^{\mathbf{b}}=\{(x,y,z)\in\mathbb{C}^{3}:\ p(a,b,c,d,x,y,z)=0\},\]
which is manifestly a cubic surface in $\mathbb{C}^{3}$. Similarly,
the parabolic character variety results from fixing conjugacy classes
in only the first 3 boundaries, so, with 
$\mathbf{h}\,=\,(h_{1},h_{2},h_{3})\,\in\, 
K^{3}$, where $K=\mathrm{SU}(2)$ and $a=\tr (h_{1})$, $b=\tr (h_{2})$ and 
$c=\tr (h_{3})$,
we obtain \[
\mathsf{X}_{\mathbf{h},0}^{G}=\{(x,y,z,d)\in\mathbb{C}^{4}:\ p(a,b,c,d,x,y,z)=0\}\]
which is a cubic 3-fold in $\mathbb{C}^{4}$, defined over $\mathbb{R}$
(since $a,b,c\in [-2,2]$ are traces of $\mathrm{SU}(2)$ matrices).}
\end{ex}

\begin{ex}
{\rm In this example, let $G=\mathrm{SL}(3,\mathbb{C})$, and $m=2$. As shown 
in \cite{La1}, $\mathrm{SL}(3,\mathbb{C})^{2}\quot \mathrm{SL}(3,\mathbb{C})$ 
is isomorphic to an degree $6$ hyper-surface in $\mathbb{C}^{9}$ which maps 
onto $\mathbb{C}^{8}$ generically two-to-one. The relative character varieties 
in this case (for both the three-holed sphere and the one-holed torus) were studied in \cite{La4}. The diagram in this case is this:
$$
\xymatrix{
\C^9\supset\mathrm{SL}(3,\mathbb{C})^{2}\quot\mathrm{SL}(3,\mathbb{C})\ar[r]^{\partial}\ar[rd]_{\partial_{par}} &(\mathrm{SL}(3,\mathbb{C})\quot \mathrm{SL}(3,\mathbb{C}))^{3}\cong \mathbb{C}^{6}\ar[d]^{\pi}\\
&(\mathrm{SL}(3,\mathbb{C})\quot\mathrm{SL}(3,\mathbb{C}))^{2}\cong \mathbb{C}^{4}
}
$$

The isomorphism $\mathrm{SL}(3,\C)\quot\mathrm{SL}(3,\C)\cong \C^2$ is given by 
$[\![g_1]\!]\longmapsto(\tr g_1,\tr g_1^{-1})$, and the map that embeds 
$\mathrm{SL}(3,\C)^2\quot\mathrm{SL}(3,\C)$ in $\C^9$ is given by 
$$[\![(g_1,g_2)]\!]\longmapsto (t_1,t_2,t_3,t_4,t_5,t_6,t_7,t_8,t_9)\, ,$$ 
where 
$t_1 
= 
\tr( g_1)$, $t_2= \tr( g_1^{-1})$, $t_3= \tr( g_2)$, $t_4= \tr( g_2^{-1})$, $t_5= \tr( g_1 g_2)$, $t_6 =\tr \left((g_1 g_2)^{-1}\right)$, $t_7= \tr( g_1^{-1}g_2)$, $t_8= \tr (g_1g_2^{-1})$, $t_9 = \tr (g_1g_2g_1^{-1}g_2^{-1})$. 

As with Example \ref{ex2}, $e_3$ is homotopic to $(e_1 e_2)^{-1}$, but unlike 
in Example \ref{ex2} the trace is not invariant under inversion, although this 
does not pose a problem since in this case we need both the boundary loop and its inverse for the boundary map.

Fix the six invariants $t_1,t_2,t_3,t_4,t_5,t_6$ corresponding to the three 
boundaries; call those fixed values 
$a_1,a_2,a_3,a_4,a_5,a_6$ respectively.

We have $$\partial^{-1}(a_1,\cdots , a_6)=\{(a_1,\cdots ,a_6, t_7,t_8,t_9)\in 
\mathbb{C}^{9} \ | \ p(a_1,\cdots ,a_6,t_7 ,t_8, t_9)=0\}$$ for $p$ a 
polynomial of degree 
$6$ in $\mathbb{C}^{9}$. Note that after fixing the 6 coordinates $p$ has 
degree 3 in the other 3 variables. Thus, $\partial^{-1}(a_1,\cdots , a_6)$ is 
isomorphic to a cubic hyper-surface in $\C^3$; as shown in \cite{La2} its smooth points form a complex symplectic surface. It likewise follows that
$$\partial_{par}^{-1}(a_1,a_2,a_3,a_4)=\{(a_1,\cdots ,a_4,t_5,t_6, 
t_7,t_8,t_9)\in 
\mathbb{C}^{9} \ | \ p(a_1,\cdots , a_4,t_5,\cdots ,t_9)=0\},$$ is isomorphic 
to a 
quartic hyper-surface in $\C^5$ since the degree of $p$ in the last 5 variables is 4.}
\end{ex}

\section*{Acknowledgments.}The first, third, and fourth authors thank Instituto Superior T\'ecnico for hospitality; their visit to IST was funded by the FCT project PTDC/MAT/099275/2008. The second author is partially supported by FCT project PTDC/MAT/120411/2010.  The third author additionally wishes to thank the GEAR network for funding his visit to the Institut Henri Poincar\'e.  The fourth author is grateful to Michel Brion for fruitful discussions. Finally, we also thank Oscar Garc\'{\i}a-Prada for communicating to us the work in \cite{BiGPM}.  Lastly, we thank the referee for useful suggestions which led us to a clarification of some of proofs.

%%%%%%%%%%%%%%%%%%%%%%%%%%%%%%%%%%%%%%%%%%%%%%%%%%%%%%%%%%%%%%

\end{document}